\documentclass[11pt]{amsart}
\usepackage{amscd,amssymb,longtable, rotating, lscape, graphicx}
\usepackage[matrix,arrow,curve]{xy}
\usepackage{supertabular}
\usepackage{setspace}

\theoremstyle{definition}
\newtheorem{theorem}[equation]{Theorem}
\newtheorem*{maintheorem*}{ Main Theorem}
\newtheorem*{theorem*}{Theorem}
\newtheorem{lemma}[equation]{Lemma}

\newtheorem{conjecture}[equation]{Conjecture}

\newtheorem{definition}[equation]{Definition}

\newtheorem*{definition*}{Definition}

\theoremstyle{remark}
\newtheorem{remark}[equation]{Remark}

\makeatletter\@addtoreset{equation}{section}
\makeatother

\makeatletter\@addtoreset{equation}{section}
\makeatletter\@addtoreset{section}{part}

\makeatother

\usepackage{xcolor}

\newcommand{\vol}{\operatorname{vol}}
\newcommand{\ord}{\operatorname{ord}}

\allowdisplaybreaks

\author{Dasol Jeong}
\author{In-Kyun Kim}
\author{Jihun Park}
\author{Joonyeong Won}

\title{New Sasaki-Einstein  $5$-manifolds}

\begin{document}

\begin{abstract}
We prove that closed simply connected $5$-manifolds $2(S^2\times S^3)\# nM_2$ allow Sasaki-Einstein structures,
where $M_2$ is  the closed simply connected $5$-manifold with $\mathrm{H}_2(M_2,\mathbb{Z})=\mathbb{Z}/2\mathbb{Z}\oplus \mathbb{Z}/2\mathbb{Z}$, $nM_2$ is the $n$-fold connected sum of $M_2$, and $2(S^2\times S^3)$ is the two-fold connected sum of $S^2\times S^3$.
\end{abstract}


\address{ \emph{Dasol Jeong}\newline \textnormal{Department of
Mathematics, POSTECH
\newline \medskip 77 Cheongam-ro, Nam-gu, Pohang, Gyeongbuk, 37673, Korea \newline
Center for Geometry and Physics, Institute for Basic Science
\newline
77 Cheongam-ro, Nam-gu, Pohang, Gyeongbuk,  37673, Korea \newline
\texttt{jdsleader@postech.ac.kr}}}

\address{ \emph{In-Kyun Kim}\newline \textnormal{Department of Mathematics, Yonsei University\newline
50 Yonsei-ro, Seodaemun-gu, Seoul, 03722,  Korea \newline
\texttt{soulcraw@gmail.com}}}

\address{ \emph{Jihun Park}\newline \textnormal{Center for Geometry and Physics, Institute for Basic Science
\newline \medskip 77 Cheongam-ro, Nam-gu, Pohang, Gyeongbuk, 37673, Korea \newline
Department of
Mathematics, POSTECH
\newline
77 Cheongam-ro, Nam-gu, Pohang, Gyeongbuk,  37673, Korea \newline
\texttt{wlog@postech.ac.kr}}}

\address{ \emph{Joonyeong Won}\newline \textnormal{Department of Mathematics, Ewha Womans University\newline
52 Ewhayeodae-gil, Seodaemun-gu, Seoul,
03760, Korea. \newline
\texttt{leonwon@kias.re.kr}}}

\maketitle

\section{Introduction}
 A Riemannian manifold $(M, g)$ is called Sasakian  if the cone metric $r^2g+ dr^2$ defines a K\"ahler metric on $M\times \mathbb{R}^+$. If the metric $g$ satisfies the Einstein condition, i.e., $\mathrm{Ric}_g=\lambda g$ for some constant $\lambda$, then the metric $g$ is called Einstein. 
 A numerous number of closed simply connected  Sasaki-Einstein manifolds, in particular $5$-manifolds, have been discovered based on the method  that was introduced by Kobayashi (\cite{Kob63}) and  developed by Boyer, Galicki, and Koll\'ar (\cite{BG00}, \cite{BGK05}, \cite{Ko05}). The upshot of their method  is briefly  presented in  \cite{Ko09} as follows. A quasi-regular Sasakian structure on a manifold $L$ can be written as the unit circle subbundle of a holomorphic Seifert $\mathbb{C}^*$-bundle over a complex algebraic orbifold $(S, \Delta)$, where $\Delta=\sum\left(1-\frac{1}{m_i}\right)D_i$, $m_i$'s are positive integers, and $D_i$'s are distinct irreducible divisors.
 A simply connected Sasakian manifold $L$ is Einstein if and only if $-(K_S+\Delta)$ is ample, the first Chern class of $c_1(L/S)$ is a rational multiple of $ -(K_S+\Delta)$, and there is an orbifold K\"ahler-Einstein metric on the orbifold $(S, \Delta)$.

Links of  quasi-homogeneous hypersurface singularities are Seifert circle bundles over the corresponding projective hypersurfaces in weighted projective spaces. For a brief explanation, we consider a quasi-smooth hypersurface  $X$ defined by a quasi-homogeneous polynomial $F(z_0,z_1,\ldots, z_n)$ in variables
$z_0, \ldots, z_n$ with weights $\mathrm{wt}(z_i)=a_i$ in a weighted projective space $\mathbb{P}(a_0, a_1, \ldots, a_n)$. The equation $F(z_0,z_1,\ldots, z_n)=0$ also defines a hypersurface $\widehat{X}$ in $\mathbb{C}^{n+1}$  that is  smooth outside the origin.  The link of $X$ is a smooth compact  manifold of real dimension $2n-1$ defined by the intersection 
\[L_X=S^{2n+1}\cap \widehat{X},\]
where $S^{2n+1}$ is the unit sphere centred at the origin in $\mathbb{C}^{n+1}$.
 Note that it is simply-connected if $n\geq 3$ (\cite[Theorem~5.2]{M68}).

Suppose that $m:=\mathrm{gcd}(a_1,\ldots, a_n)>1$ and $\mathrm{gcd}(a_0, a_1, \ldots, a_{i-1}, \widehat{a_i}, a_{i+1}, \ldots, a_n)=1$ for each  $i=1,\ldots, n.$  Set $b_0=a_0$ and $b_i=\frac{a_i}{m}$ for $i=1,\ldots, n.$     Then, the weighted projective space $\mathbb{P}(a_0, a_1, \ldots, a_n)$ is not well-formed, while the weighted projective space $\mathbb{P}(b_0, b_1, \ldots, b_n)$ is well-formed (see \cite[Definition~5.11]{IF00}).
There is a quasi-homogeneous polynomial $G(x_0,\ldots, x_n)$ in variables
$x_0, \ldots, x_n$ with weights $\mathrm{wt}(x_i)=b_i$ such that $F(z_0,z_1,\ldots, z_n)=G(z_0^m,z_1,\ldots, z_n).$ The equation $G(x_0,\ldots,x_n)=0$ defines a quasi-smooth hypersurface $Y$ in $\mathbb{P}(b_0, b_1, \ldots, b_n)$. 
We suppose that $\deg_{\mathbf{w}}(F)-\sum a_i<0$ and $Y$ is well-formed in $\mathbb{P}(b_0, b_1, \ldots, b_n)$  (see \cite[Definition~6.9]{IF00}).   Denote by $D$ the divisor on~$Y$ cut by $x_0=0$.  
We may consider the log pair $(Y,\frac{m-1}{m}D)$ as a Fano orbifold.
The method by Kobayashi has evolved into the following assertion through the works of Boyer, Galicki, and Koll\'ar

\begin{theorem}[{\cite[Theorem~2.1]{BG00}, \cite[Theorem~5]{Kob63}}]\label{theorem:BGK-method}
If $(Y,\frac{m-1}{m}D)$ allows an orbifold  K\"ahler-Einstein metric, then there is a Sasaki-Einstein metric on the link $L_X$ of~$X$.
\end{theorem}

 Closed simply connected $5$-manifolds are completely classified by Barden and Smale (\cite{Ba65}, \cite{SM62}).  In particular, Smale  has classified all the closed simply connected spin $5$-manifolds (\cite{SM62}), which are called Smale $5$-manifolds. For a positive integer $m$,  up to diffeomorphisms, there is a unique closed simply connected  spin $5$-manifold $M_m$ with $\mathrm{H}_2(M_m,\mathbb{Z})=\mathbb{Z}/m\mathbb{Z}\oplus \mathbb{Z}/m\mathbb{Z}$. Furthermore, 
a closed simply connected spin $5$-manifold $M$   is of the form
\[M=kM_{\infty}\#M_{m_1}\#\ldots\# M_{m_r},\]
where  $kM_\infty$ is the $k$-fold connected sum of $S^2\times S^3$ for a  non-negative integer $k$ and $m_i$ is a positive integer greater than 1 with $m_i$ dividing $m_{i+1}$.

Many efforts have been made to classify all the closed simply connected Sasaki-Einstein $5$-manifolds.  
To be precise,
 for each Smale  $5$-manifold (every Sasaki-Einstein manifolds are spin), we want to determine whether it  has a quasi-regular Sasaki-Einstein  structure or not. 
Such efforts and their results are summarized in 
\cite{JP21}. Toward complete classification,  three conjectures were proposed in \cite{JP21}.  One of them is

\begin{conjecture}\label{conjecture:k-n2}
For each integer $k\leq 8$ and $n\geq 2$, the Smale $5$-manifold $kM_\infty\#nM_2$ admits a Sasaki-Einstein metric.
\end{conjecture}

The conjecture has been verified for $k=0, 1$ so far (\cite{Ko05}, \cite{PW19}). Also,  $2M_\infty\#nM_2$ is proven to allow a Sasakian metric of positive Ricci curvature (\cite[Theorem~B]{BN10}).
In this article,  we prove the conjecture for $k=2$.
\begin{maintheorem*}
For every positive integer $n$, the Smale manifold $2M_\infty\#nM_2$ allows a Sasaki-Einstein metric.
\end{maintheorem*}

\section{K\"ahler-Einstein metric and K-stability} 
The theory  on K\"ahler-Einstein metrics and K-stability of Fano varieties and the theory on valuative criterions for K-stability  have developed dramatically for the last ten years.

In 2016 Fujita and Odaka introduced a new invariant of a Fano variety, so-called  $\delta$-invariant, which has evolved into a criterion for K-stability through the work of Blum and Jonsson. The $\delta$-invariant measures how singular the average divisors of sections that form a basis for plurianticanonical linear systems are, using their log canonical thresholds.

Let $X$ be a projective $\mathbb{Q}$-factorial normal variety and $\Omega$ be a $\mathbb{Q}$-divisor on $X$ such that the log pair $(X, \Omega)$ has at worst Kawamata log terminal singularities.
We suppose that $(X, \Omega)$ is a log  $\mathbb{Q}$-Fano variety, i.e., the divisor $-(K_X+\Omega)$ is ample.

\begin{definition}\label{definition:basis-type}
Let $m$ be a positive integer such that $|-m(K_X+\Omega)|$ is non-empty. Set $\ell_m=h^0(X,\mathcal{O}_X(-m(K_X+\Omega))$.  For a section $s$ in $\mathrm{H}^0(X,\mathcal{O}_X(-m(K_X+\Omega)))$, we denote the effective divisor of the section $s$ by $D(s)$.
If $\ell_m$ sections $s_1,\ldots, s_{\ell_m}$  form  a basis of  the space $\mathrm{H}^0(X,\mathcal{O}_X(-m(K_X+\Omega))$, then the  anticanonical 
$\mathbb{Q}$-divisor 
\[D:=\frac{1}{\ell_m}\sum_{i=1}^{\ell_m}\frac{1}{m}D (s_i)\]
is said to be of $m$-basis type.
We set 
\[\delta_m(X, \Omega)=\mathrm{sup}\left\{\lambda\in\mathbb{Q}\ \left|\ %
\aligned
&\text{the  log pair}\ \left(X, \Omega+\lambda D \right)\ \text{is log canonical for }\\
&\text{every effective $\mathbb{Q}$-divisor $D$ of $m$-basis type}\\
\endaligned\right.\right\}.\]
The $\delta$-invariant of $(X, \Omega)$ is defined by the number
\[\delta(X, \Omega)=\limsup_m \delta_m(X,\Omega).\]
\end{definition}

To study the $\delta$-invariant from local viewpoints, we set
\[\delta_{Z,m}(X, \Omega)=\mathrm{sup}\left\{\lambda\in\mathbb{Q}\ \left|\ %
\aligned
&\text{the  log pair}\ \left(X, \Omega+\lambda D \right)\ \text{is log canonical along $Z$  }\\
&\text{for every effective $\mathbb{Q}$-divisor $D$ of $m$-basis type}\\
\endaligned\right.\right\}\]
for a closed subvariety $Z$ of $X$.
The local $\delta$-invariant of $(X, \Omega)$ along $Z$ is defined by the number
\[\delta_Z(X, \Omega)=\limsup_m \delta_{Z,m}(X,\Omega).\]

Using the $\delta$-invariant, Blum-Jonsson (\cite{Blum-Jonsson17}) and Fujita-Odaka (\cite{Fujita-Odaka16})  set up a  criterion for K-(semi)stability in an algebro-geometric way. Due to the result \cite[Theorem~1.5]{LXZ21}, the criterion reads as follows:
\begin{theorem}\label{theorem:delta}
A log $\mathbb{Q}$-Fano variety $(X, \Omega)$ is K-stable (resp. K-semistable) if and only if $\delta(X, \Omega)>1$ (resp. $\geq 1$).
\end{theorem}

The bridge between K-polystability and existence of K\"ahler-Einstein metrics  has been completely established for log Fano pairs (\cite{B16}, \cite{BBJ}, \cite{LXZ21},  \cite{CDS1, CDS2, CDS3}, \cite{Tian15}, \cite{Li19}, \cite{LTW19}, \cite{Xu21}). 
\begin{theorem}\label{theorem:delta-KE}
A Fano orbifold $(X, \Omega)$ is K-polystable if and only if it allows an orbifold K\"ahler-Einstein metric.
\end{theorem}

\section{Tools for $\delta$-invariant }

Let $S$ be a surface with at most cyclic quotient singularities and  $D$ an effective $\mathbb{Q}$-divisor on the surface~$S$. Also let $p$ be a point of  $S$.

\begin{lemma}
\label{lemma:mult1}
Suppose that $p$ is a smooth point of $S$.
If the log pair $(S,D)$ is not log canonical at  $p$, then $\mathrm{mult}_p(D)>1$.
\end{lemma}
\begin{proof}
See \cite[Proposition~9.5.13]{La04}, for instance.
\end{proof}

Let $C$ be an integral curve on $S$ that passes through the point $p$.
Suppose that $C$ is not contained in the support of the divisor $D$.
If $p$ is a smooth point of the surface $S$ and
the log pair $(S,D)$ is not log canonical at $p$, then 
it follows from Lemma~\ref{lemma:mult1} that $D\cdot C>1$.

\begin{lemma}
\label{lemma:mult-1-n}
Suppose that $p$ is a  cyclic quotient singularity of type $\frac{1}{n}(a,b)$, where $a$ and $b$ are coprime positive integers that are also coprime to $n$.
If the log pair $(S,D)$ is not log canonical at~$p$
and $C$ is not contained in the support of the divisor $D$, then $$D\cdot C>\frac{1}{n}.$$
\end{lemma}
\begin{proof}
This follows from \cite[Proposition~3.16]{Ko97}, Lemma~\ref{lemma:mult1}, and \cite[Lemma~2.2]{CPS10}.
\end{proof}
In general, the curve $C$ may be contained in the support of the divisor $D$. In this case, we write
$$
D=aC+\Omega,
$$
where $a$ is a positive rational number and $\Omega$ is an effective $\mathbb{Q}$-divisor on $S$ whose support does not contain the curve $C$.
We suppose that $(S, C)$ is purely log terminal around $p$.

\begin{lemma}
\label{lemma:inversion-of-adjunction}
Suppose that $a\leq1$ and the log pair $(S,D)$ is not log canonical at $p$.
\begin{enumerate}
\item If $p$ is a smooth point,  then
$$
C\cdot\Omega\geq\left(C\cdot\Omega\right)_p>1,
$$
where $\left(C\cdot\Omega\right)_p$ is the local intersection number of $C$ and $\Omega$ at $p$.
\item  If $p$ is a  cyclic quotient singularity of type $\frac{1}{n}(a,b)$,
then
$$
C\cdot\Omega>\frac{1}{n}.
$$
\end{enumerate}
\end{lemma}

\begin{proof}
See \cite[Lemma~2.5]{CPS10}.
\end{proof}

We now let  $\left(S, \Omega\right)$ be a log del Pezzo surface that allows only Kawamata log terminal singularities.
It follows from \cite[Corollary~1.3.2]{BCHM} that a log del Pezzo surface is a Mori dream space.
Let $C$ be a prime divisor over $S$ and let $\pi:\hat{S}\to S$ be a birational morphism such that $C$ is a divisor on~$\hat{S}$.

We first set
\begin{equation}\label{equation:S}
S_{S, \Omega}(C)=\frac{1}{\left(-\left(K_S+\Omega\right)\right)^2}\int_0^{\infty}\vol\left(\pi^*\left(-\left(K_S+\Omega\right)\right)-tC\right)dt.
\end{equation}

Taken the definition of the basis type divisors in Definition~\ref{definition:basis-type} into consideration, it is natural to expect that a divisor of $m$-basis type cannot carry  a  prime divisor with big multiplicity for a sufficiently large $m$. Indeed, the following bound is  originally given in \cite[Lemma~2.2]{Fujita-Odaka16}.
\begin{lemma}\label{lemma:Fujita}
For a given real number $\epsilon>0$, there is an integer $\mu$ such that 
whenever $m>\mu$, we have
\[\ord_{C}(\pi^*(D))\leq S_{S, \Omega}(C)+\epsilon \]
for every ample $\mathbb{Q}$-divisor $D$ of $m$-basis type with respect to $(S, \Omega)$.
\end{lemma}
\begin{proof}
See  \cite[Theorem~2.9]{CPS21}.
\end{proof}

The following describe how to estimate $\delta(S, \Omega)$ from local viewpoints, which are developed in \cite{Abban-Zhuang20} and simplified in \cite{Fano3-folds} and \cite{Fujita21}.

Suppose that $(S,\Omega+C)$ is purely log terminal.
Let $\tau$ be  the supremum of the positive real numbers $t$ such that $-\left(K_S+\Omega\right)-tC$ is big.
For  a real number $u\in (0,\tau)$,   we write  the Zariski decomposition of  $-(K_S+\Omega)-uC$ as
\[-(K_S+\Omega)-uC=P(u)+N(u),\]
where
$P(u)$ and $N(u)$ are the positive and the negative parts, respectively.
Let $p$ be a point on~$C$. Define
\begin{equation}\label{equation:h}
h(u)=\left(P(u)\cdot C\right)\cdot\ord_{p}\left(N(u)\big\vert_{C}\right)+\int_0^\infty \vol\left(P(u)\big\vert_C-vp\right)dv
\end{equation}
and then put
\begin{equation}\label{equation:Kento-formula-S}
S\left(W^C_{\bullet,\bullet}; p\right)=\frac{2}{(-(K_S+\Omega))^2}\int_0^\tau h(u)du.
\end{equation}

 Recall that we have the following adjunction formula

$$
\left(K_S+\Omega+C\right)\big\vert_C=K_C+\Delta
$$ where $\Delta$ is the different for $K_S+\Omega+C$. Then,
the log discrepancy of the pair $(C, \Delta)$ along the divisor $p$ is

$$
A_{C, \Delta}(p)=1-\ord_p(\Delta).
$$ 

If $p$ is a quotient singular point of type $\frac{1}{n}(a, b)$, then
\begin{equation}\label{equation:dicr}
A_{C, \Delta}(p)=\frac{1}{n}-\left(\Omega\cdot C\right)_p.
\end{equation}

\begin{theorem}\label{theorem:Hamid-Ziquan-Kento}
The local $\delta$-invariant of $(S, \Omega)$ at the point $p$ satisfies the inequality
$$
\delta_p(S, \Omega)\geq\mathrm{min}\left\{\frac{1}{S_{S,\Omega}(C)},\frac{A_{C,\Delta}(p)}{S\left(W^C_{\bullet,\bullet};p\right)}\right\}.
$$
\end{theorem}
\begin{proof}
This immediately follows from   \cite[Theorem~4.8~(2) and Corollary~4.9]{Fujita21} because the point~$p$ is the only prime divisor over the curve $C$ with the center at $p$  in \cite[Definition~3.11]{Fujita21}.
\end{proof}

\section{Sasaki-Einstein $5$-manifolds  $2M_\infty\#nM_2$}

For each integer $n\geq 2$, let $\widehat{S}_n$ be a quasi-smooth hypersurface of degree $4(2n+1)$ in  $\mathbb{P}(2, 4, 4n, 4n+1)$. This hypersurface appears in \cite{BN10}  to give a Sasakian metric of positive Ricci curvature to $2M_\infty\#nM_2$. By using appropriate coordinate changes, we may assume that the surface is defined by 
\[w^2x+yz(z-y^n)+zx\widehat{A}_{4n+2}(x,y)+x\widehat{A}_{8n+2}(x,y)=0,\]
where $x$, $y$, $z$, $w$  are quasi-homogeneous coordinates  with $\mathrm{wt}(x)=2$,  $\mathrm{wt}(y)=4$,  $\mathrm{wt}(z)=4n$,  $\mathrm{wt}(w)=4n+1$ and
$\widehat{A}_{4n+2}(x,y)$, $\widehat{A}_{8n+2}(x,y)$ are quasi-homogeneous polynomials  of degrees $4n+2$ and $8n+2$, respectively, in $x$, $y$.

We use the same notation $x$, $y$, $z$, $w$ for homogeneous coordinates of the weighted projective space $\mathbb{P}(1, 2, 2n, 4n+1)$  with $\mathrm{wt}(x)=1$,  $\mathrm{wt}(y)=2$,  $\mathrm{wt}(z)=2n$,  $\mathrm{wt}(w)=4n+1$.  
Let $S_n$ be the quasi-smooth hypersurface of degree $2(2n+1)$ in  $\mathbb{P}(1, 2, 2n, 4n+1)$
defined by 
\[wx+yz(z-y^n)+zxA_{2n+1}(x,y)+xA_{4n+1}(x,y)=0,\]
where $A_{2n+1}(x, y)$ and $A_{4n+1}(x, y)$  are the quasi-homogeneous polynomials  of degrees $2n+1$ and $4n+1$ defined by $\widehat{A}_{4n+2}(x,y)$ and  $\widehat{A}_{8n+2}(x,y)$, respectively, with weights $\mathrm{w}(x)=1$ and $\mathrm{w}(y)=2$. We denote by $W$ the irreducible divisor on $S_n$ cut by $w=0$. As an orbifold, $\widehat{S}_n$ can be regarded as the log del Pezzo surface
$\left(S_n,\frac{1}{2}W\right)$.

We consider the link of $\widehat{S}_n$.  It follows from \cite[Corollary]{MO70} that the link of $\widehat{S}_n$ has the second Betti number  $2$.  The curve $W$ is isomorphic to a smooth curve of degree $2n+1$ in $\mathbb{P}(1,1,n)$, and hence its genus is~$n$.  It then follows from \cite[Theorem~5.7]{Ko05} that the torsion part of the second homology group of the link is  $\left(\mathbb{Z}/2\mathbb{Z}\oplus \mathbb{Z}/2\mathbb{Z}\right)^{\oplus n}$.  Consequently, the link of $\widehat{S}_n$ is diffeomorphic to $2M_\infty\#nM_2$ by  \cite[Theorem]{SM62}.   

Therefore, Theorem~\ref{theorem:BGK-method} implies that  the following statement guarantees existence of a Sasaki-Einstein metric on $2M_\infty\#nM_2$.  In other words, Main Theorem  immediately  follows from the theorem below.

\begin{theorem}\label{theorem:KE-del}
For $n\geq 2$, $\left(S_n,\frac{1}{2}W\right)$ allows an orbifold K\"ahler-Einstein metric. 
\end{theorem}

We remark here that $2M_\infty\#M_2$ is  already shown to admit a Sasaki-Einstein metric (\cite[Theorem~A]{BN10}).

\section{Proof of Theorem~\ref{theorem:KE-del}}

Due to Theorems~\ref{theorem:delta} and~\ref{theorem:delta-KE}, in order to prove Theorem~\ref{theorem:KE-del}, it is enough to show that
\[\delta\left(S_n,\frac{1}{2}W\right)>1.\]
In this section, we achieve this inequality by verifying 
 \[\delta_p\left(S_n,\frac{1}{2}W\right)\geq\frac{20n + 5}{20n +4}\]
 for each point $p$ in $S_n$.

The surface $S_n$ has four distinct singular points at $O_w=[0:0:0:1]$, $O_z=[0:0:1:0]$, $O_0=[0:1:0:0]$, and $O_1=[0:1:1:0]$.
These are all cyclic quotient singularities of types $\frac{1}{4n+1}(1, n)$, $\frac{1}{2n}(1,1)$, $\frac{1}{2}(1,1)$, and $\frac{1}{2}(1,1)$, respectively. 

Denote by $C_x$  the divisor on $S_n$ cut by $x=0$.
The divisor $C_x$ consists of three components. To be precise,
\[C_x=L_{xy}+R_0+R_1,\]
where $L_{xy}$ is defined by $x=y=0$, $R_0$ by $x=z=0$, and $R_1$ by $x=z-y^n=0$. 
Each pair of these three curves meet only at $O_w$.

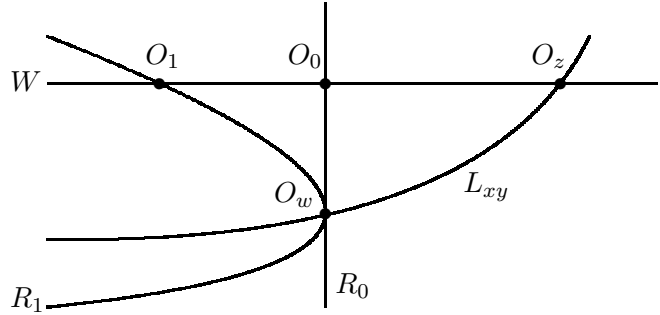
\begin{center}
\begin{minipage}[m]{.25\linewidth}\setlength{\unitlength}{.90mm}
\begin{center}
\begin{picture}(100,50)(25,5)
\thicklines

\qbezier(90,50)(75,19)(10,20)

\qbezier(10,10)(92,18)(10,50)

\put(10,43){\line(1,0){90}}
\put(51,10){\line(0,1){45}}

\put(3,42){\mbox{ $W$}}
\put(52.5,12){\mbox{$R_{0}$}}

\put(3,10){\mbox{ $R_1$}}
\put(70,27){\mbox{ $L_{xy}$}}

\put(48.5,22.7){\mbox{ $\bullet$}}
\put(42,25.5){\mbox{ $O_w$}}

\put(48.5,42){\mbox{ $\bullet$}}
\put(44,46){\mbox{ $O_0$}}

\put(24,42){\mbox{ $\bullet$}}
\put(23,46){\mbox{ $O_1$}}

\put(83.2,42){\mbox{ $\bullet$}}
\put(80,46){\mbox{ $O_z$}}

\end{picture}
\end{center}
\end{minipage}
\end{center}

Their intersection numbers on $S_n$ are as follows:
\[
\label{equation:intersection-number}
\begin{split}
&L_{xy}^2=-\frac{4n-1}{2n(4n+1)},\ \ R_0^2=R_1^2=-\frac{2n+1}{2(4n+1)},\\ 
&L_{xy}\cdot R_{0}=L_{xy}\cdot R_{1}=\frac{1}{4n+1},\ \
 R_0\cdot R_{1}=\frac{n}{4n+1}.\\
 \end{split}
\]

The divisor  $-\left(K_{S_n}+\frac{1}{2}W\right)$ is  equivalent to $\frac{3}{2}C_x$ and its self-intersection number is 
\[\left(K_{S_n}+\frac{1}{2}W\right)^2=\frac{9(2n+1)}{8n(4n+1)}.\]

The irreducible curves $L_{xy}$, $R_0$, $R_1$ belong to the boundary of the pseudoeffective cone of $S_n$ since they are of negative self-intersection.
Therefore, for $t >\frac{3}{2}$, the divisor
\[\frac{3}{2}C_x-t L_{xy}= \left(\frac{3}{2}-t\right)L_{xy}+\frac{3}{2}R_0+\frac{3}{2}R_1\]
is not pseudoedffective. 
Set
\begin{equation}\label{equation:Zariski-L}
\aligned
&P_L(t):=\left\{%
\aligned
&\frac{3}{2}C_x-t L_{xy} \ \ \ \text{ for $0\leq t \leq\frac{3}{4}$},\\%
&\\
& \frac{3}{2}C_x-t L_{xy}-\frac{4t -3}{2}(R_0+R_1)
\ \ \ \text{ for $\frac{3}{4}\leq t \leq\frac{3}{2}$},\\
\endaligned\right.\\
&\\
&N_L(t):=\left\{%
\aligned
&0 \ \ \ \text{ for $0\leq t \leq\frac{3}{4}$},\\%
&\\
&\frac{4t -3}{2}(R_0+R_1) \ \ \ \text{ for $\frac{3}{4}\leq t \leq\frac{3}{2}$}.\\%
\endaligned\right.\\
\endaligned
\end{equation}

For each $i=0, 1$,
\[\left(\frac{3}{2}C_x-t L_{xy}\right)\cdot R_i=\frac{3-4t}{4(4n+1)}.\]
This implies that the divisor $P_L(t)$ is nef for $0\leq t \leq\frac{3}{4}$.  
Furthermore,  for $\frac{3}{4}\leq t \leq\frac{3}{2}$, $P_L(t)$
is a nef divisor with $P_L(t)\cdot R_0=P_L(t)\cdot R_1=0$.  Consequently, the Zariski decomposition of $\frac{3}{2}C_x-t L_{xy}$ is given by
\[\frac{3}{2}C_x-t L_{xy}=P_L(t)+N_L(t)\]
for $0 \leq t \leq\frac{3}{2}$.
We then see that the volume of $\frac{3}{2}C_x-t L_{xy}$ is
\begin{equation}\label{equation:volume-L}
\vol\left(\frac{3}{2}C_x-t L_{xy}\right)=\left\{%
\aligned
& -\frac{4(4n-1)t^2+12t-9(2n+1) }{8n(4n+1)}\ \ \ \text{ for $0\leq t \leq\frac{3}{4}$},\\%
&\\
& \frac{(3-2t)^2}{8n} \ \ \ \text{ for $\frac{3}{4}\leq t \leq\frac{3}{2}$},\\%
&\\
& 0  \ \ \ \text{ for $t \geq\frac{3}{2}$},\\
\endaligned\right.%
\end{equation}
and the function in \eqref{equation:S} is given by
\begin{equation}\label{equation:S-L}
    S_{S_n,\frac{1}{2}W}(L_{xy})=\frac{3n+1}{2(2n+1)}.
\end{equation}

Moreover, note that
\[
    \ord_p\left(N_L(t)\big\vert_{L_{xy}}\right)=0,
\]
for $p\in L_{xy}\setminus\{O_w\}$. Then, the function in \eqref{equation:h} is given by
\[\begin{split}
h_{p,L}(t) &=\int_0^\infty \vol\left(P_L(t)\big\vert_{L_{xy}}-vp\right)dv\\
&=\frac{1}{2}\left(P_L(t)\cdot L_{xy}\right)^2\\
&=\left\{\aligned
&\frac{1}{32n^2(4n+1)^2}\left(2(4n-1)t+3\right)^2 \ \ \ \text{ for $0\leq t\leq \frac{3}{4}$}\\
&\\
&\frac{1}{32n^2}(-2t+3)^2\ \ \ \text{ for $\frac{3}{4}\leq t\leq \frac{3}{2}$,}
\endaligned\right.
\end{split}\]
and the value in \eqref{equation:Kento-formula-S} for $p\in L_{xy}\setminus\{O_w\}$ is given by
\begin{equation}\label{equation:Kento-formula-S-L}
\begin{split}
    S(W_{\bullet,\bullet}^{L_{xy}};p)&=\frac{16n(4n+1)}{9(2n+1)}\int_0^\frac{3}{2}h_{p,L}(t)dt\\
    &=\frac{4n+1}{18n(2n+1)}\left\lbrace\int_0^\frac{3}{4}\frac{1}{(4n+1)^2}(2(4n-1)t+3)^2dt+\int_\frac{3}{4}^\frac{3}{2}(3-2t)^2dt\right\rbrace\\
    &=\frac{4n^2+3n+1}{4n(2n+1)(4n+1)}.
\end{split}\end{equation}

We now use indices $i, j $ such that  $\{i, j\}=\{0, 1\}$.
For $t>\frac{3}{2}$, the divisor
\[\frac{3}{2}C_x-t R_i= \frac{3}{2}L_{xy}+\left(\frac{3}{2}-t\right)R_i+\frac{3}{2}R_j\]
is not pseudoedffective.
Put
\begin{equation}\label{equation:Zariski-Ri}
\aligned
&P_{R_i}(t):=\left\{%
\aligned
&\frac{3}{2}C_x-t R_i\ \ \ \text{ for $0\leq t \leq\frac{3}{4n}$},\\%
&\\
&\frac{3}{2}C_x-t R_i-\frac{4nt -3}{2(2n-1)}(L_{xy}+R_j)
\ \ \ \text{ for $\frac{3}{4n}\leq t \leq\frac{3}{2}$}.\\%
\endaligned\right.\\
&\\
&N_{R_i}(t):=\left\{%
\aligned
&0 \ \ \ \text{ for $0\leq t \leq\frac{3}{4n}$},\\%
&\\
&\frac{4nt -3}{2(2n-1)}(L_{xy}+R_j) \ \ \ \text{ for $\frac{3}{4n}\leq t \leq\frac{3}{2}$}.\\%
\endaligned\right.\\
\endaligned
\end{equation}
We have
\[\left(\frac{3}{2}C_x-t R_i\right)\cdot L_{xy}=\frac{3-4nt}{4n(4n+1)}, \ \ \ \left(\frac{3}{2}C_x-t R_i\right)\cdot R_j=\frac{3-4nt}{4(4n+1)},\]
and hence we see that $P_{R_i}(t)$ is nef for $0\leq t \leq\frac{3}{4n}$.
 The divisor $P_{R_i}(t)$ is a nef divisor with $P_{R_i}(t)\cdot L_{xy}=P_{R_i}(t)\cdot R_j=0$ for  $\frac{3}{4n}\leq t \leq\frac{3}{2}$.
Therefore, the Zariski decomposition of $\frac{3}{2}C_x-t R_i$ is given by \eqref{equation:Zariski-Ri}.

Consequently, the volume is given by
\begin{equation}\label{equation:Volume-Ri}
\vol\left(\frac{3}{2}C_x-t R_i\right)=\left\{%
\aligned
&-\frac{4n(2n+1)t^2+12nt-9(2n+1)}{8n(4n+1)} \ \ \ \text{ for $0\leq t \leq\frac{3}{4n}$},\\%
&\\
& \frac{(3-2t)^2}{8(2n-1)} \ \ \ \text{ for $\frac{3}{4n}\leq t \leq\frac{3}{2}$},\\%
&\\
& 0  \ \ \ \text{ for $t \geq\frac{3}{2}$},\\
\endaligned\right.%
\end{equation}
and the function in \eqref{equation:S} is given by
\begin{equation}\label{equation:S-Ri}
    S_{S_n,\frac{1}{2}W}(R_i)=\frac{4n^2+3n+1}{4n(2n+1)}.
\end{equation}
For a point $p\in R_i\setminus\{O_w\}$, note that\[
    \ord_p\left(N_{R_i}(t)|_{R_i}\right)=0
\]
on $0\leq t\leq\frac{3}{2}$. Thus, for $p\in R_i\setminus\{O_w\}$ the function in \eqref{equation:h}  is given by
\[
\begin{split}
h_{p,R_i}(t)&=\int_0^\infty \vol\left(P_{R_i}(t)\big\vert_{R_i}-vp\right)dv\\
&=\frac{1}{2}(P_{R_i}(t)\cdot R_i)^2\\
&=\left\{\aligned 
&\frac{1}{2(4n+1)^2}\left(\left(\frac{1}{2}+n\right)t+\frac{3}{4}\right)^2\ \ \ \text{ for $0\leq t \leq\frac{3}{4n}$}\\
&\\
&\frac{1}{32(2n-1)^2}\left(-2t+3\right)^2 \ \ \ \text{ for $\frac{3}{4n}\leq t \leq\frac{3}{2}$,}\\
\endaligned\right.\\
\end{split}\]
and the value in \eqref{equation:Kento-formula-S} is given by
\begin{equation}\label{equation:Kento-formula-S-Ri}\begin{split}
    S(W^{R_i}_{\bullet,\bullet};p)&=\frac{16(4n+1)}{9(2n+1)}\int_0^\frac{3}{2}h_{p,R_i}(t)dt\\
    &=\frac{16n(4n+1)}{9(2n+1)}\Bigg\{\int_0^\frac{3}{4n}\frac{1}{(2(4n+1)^2}\left(\left(\frac{1}{2}+n\right)t+\frac{3}{4}\right)^2dt\\
    &\qquad\qquad\qquad\qquad\qquad\qquad+\int_\frac{3}{4n}^\frac{3}{2}\frac{1}{32(2n-1)^2}(-2t+3)^2dt\Bigg\}\\
    &=\frac{8n^2+7n+1}{8n(2n+1)(4n+1)}.
\end{split}\end{equation}

Let $C_\gamma$ be the curve on $S_n$ cut by $y=\gamma x^2$ for a constant $\gamma$. It consists of two irreducible curves. One is $L_{xy}$ and the other is the curve $R$ defined by $$y-\gamma x^2=w+\gamma xz(z-\gamma^nx^{2n})+zA_{2n+1}(x, \gamma x^2)+A_{4n+1}(x, \gamma x^2)=0.$$
Their intersection numbers are as follows:
\[R^2=\frac{1}{2n}, \ \ \ L_{xy}\cdot R=\frac{1}{2n}.\]
Also we see that $W$ and $R$ meets at $O_z$ with local intersection number $\frac{1}{2n}$ and
\[W\cdot R=2+\frac{1}{2n}.\]
Besides the singular point $O_z$, the curve $R$ meets $W$ either transversally at two distinct smooth points or  tangentially at a single smooth point with local intersection number $2$.

Since $L_{xy}$ is of negative self-intersection, $\frac{3}{4}C_\gamma-tR$ is not pseudoeffective for $t>\frac{3}{4}$. Put
\begin{equation}\label{equation:Zariski-R}
\aligned
&P_R(t):=\left\{
\aligned
& \frac{3}{4}C_\gamma-tR\ \ \ \text{ for $0\leq t\leq\frac{3}{2(4n+1)}$},\\
&\\
&\frac{3}{4}C_\gamma-tR-\frac{2(4n+1)t-3}{2(4n-1)}L_{xy}
\ \ \ \text{ for $\frac{3}{2(4n+1)}\leq t\leq \frac{3}{4}$}.\\
\endaligned\right.\\
&\\
&N_R(t):=\left\{
\aligned
&0 \ \ \ \text{ for $0\leq t\leq\frac{3}{2(4n+1)}$},\\
&\\
&\frac{2(4n+1)t-3}{2(4n-1)}L_{xy} \ \ \ \text{ for $\frac{3}{2(4n+1)}\leq t\leq \frac{3}{4}$}.\\
\endaligned\right.\\
\endaligned
\end{equation}
Then, we have
\[
    P_R(t)\cdot L_{xy}=\left\{\aligned
    &\frac{1}{2n}\left(\frac{3}{2(4n+1)}-t\right)\ \ \ \text{ for $0\leq t\leq\frac{3}{2(4n+1)}$},\\
    &\\
    &0\ \ \ \text{ for $\frac{3}{2(4n+1)}\leq t\leq\frac{3}{4}$},\\
    \endaligned\right.
\]
and hence we see that $P_R(t)$ is nef for $0\leq t\leq\frac{3}{4}$. Consequently, the Zariski decomposition of $\frac{3}{4}C_\gamma-tR$ is given by \eqref{equation:Zariski-R}
for $0\leq t\leq\frac{3}{4}$. Thus, the volume  is given by

\begin{equation}\label{equation:volume-R}
\vol\left(\frac{3}{2}C_\gamma-t R\right)=\left\{%
\aligned
&\frac{4(4n+1)t^2-12(4n+1)t+9(2n+1)}{8n(4n+1)} \ \ \ \text{ for $0\leq t\leq\frac{3}{2(4n+1)}$},\\%
&\\
& \frac{(3-4t)^2}{4(4n-1)} \ \ \ \text{ for $\frac{3}{2(4n+1)}\leq t\leq\frac{3}{4}$},\\%
&\\
& 0  \ \ \ \text{ for $t \geq\frac{3}{4}$},\\
\endaligned\right.%
\end{equation}
and the value in \eqref{equation:S} is given by

\begin{equation}\label{equation:S-R}
    S_{S_n,\frac{1}{2}W}(R)=\frac{16n^3+16n^2+7n+1}{2(2n+1)(4n+1)^2}.
\end{equation}

We now consider an effective $\mathbb{Q}$-divisor $D$  numerically equivalent to $-\left(K_{S_n}+\frac{1}{2}W\right)$. We may write
\begin{equation}\label{equation:D-ab}
D=aW+bL_{xy}+b_0R_0+b_1R_1+\Delta,
\end{equation}
where $a$, $b$, and $b_i$ are non-negative rational numbers and $\Delta$ is an effective $\mathbb{Q}$-divisor whose support contains none of the curves $W$, $L_{xy}$, $R_0$, $R_1$.
Also we may write
\begin{equation}\label{equation:D-c}
D=aW+bL_{xy}+cR+\Omega,
\end{equation}
where $\Omega$ is an $\mathbb{Q}$-effective divisor whose support contains none of $W$, $L_{xy}$, $R$.

\begin{lemma}\label{lemma:estimate-ab} For a sufficiently large integer $m$, suppose that $D$ is of $m$-basis type  with respect to the log del Pezzo surface $\left(S_n,\frac{1}{2}W\right)$.
Then
 $$a<\frac{1}{8n}; \ \ b<\frac{3n+2}{2(2n+1)}; \ \  b_0, b_1 <\frac{8n^2+6n+3}{8n(2n+1)}; \ \ c<\frac{3}{10}.$$
\end{lemma}
\begin{proof}
The first inequality immediately follows from 
\[\frac{1}{D^2}\int^{\infty}_0\vol\left(D-t W\right)dt=
\int^{\frac{3}{2(4n+1)}}_0\left(1-\frac{2(4n+1)}{3}t\right)^2dt=\frac{1}{2(4n+1)}<\frac{1}{8n}\]
via Lemma~\ref{lemma:Fujita}.

Also, it follows from  \eqref{equation:S-L},  \eqref{equation:S-Ri}, \eqref{equation:S-R}, and Lemma~\ref{lemma:Fujita}  that 
\[\begin{split}
&b\leq S_{S_n,\frac{1}{2}W}(L_{xy})+\epsilon=\frac{3n+1}{2(2n+1)}+\epsilon<\frac{3n+2}{2(2n+1)},\\
&b_i\leq S_{S_n,\frac{1}{2}W}(R_i)+\epsilon=\frac{4n^2+3n+1}{4n(2n+1)}+\epsilon<\frac{8n^2+6n+3}{8n(2n+1)},\\
&c\leq S_{S_n,\frac{1}{2}W}(R)+\epsilon=\frac{16n^3+16n^2+7n+1}{2(2n+1)(4n+1)^2}+\epsilon<\frac{3}{10},
\end{split}\]
where $\epsilon$ is a sufficiently small positive rational number.
\end{proof}

From now on, we put $\lambda = \frac{20n + 5}{20n +4}$.

\begin{theorem}\label{theorem:smooth}
For a smooth point $p$,
 $$\delta_p\left(S_n, \frac{1}{2}W\right)\geq \lambda.$$ \end{theorem}
\begin{proof}
With a sufficiently large positive integer $m$, let $D$ be a $\mathbb{Q}$-divisor of $m$-basis type with respect to the log del Pezzo surface $(S_n, \frac{1}{2}W)$. It is enough to show that the log pair 
$$\left(S_n, \frac{1}{2}W+\lambda D\right)$$
is log canonical on the smooth locus of $S_n$.

Suppose that the log pair 
is not log canonical at a smooth point $p$.

We write the divisor $D$ as \eqref{equation:D-ab}, i.e., 
\[D=aW+bL_{xy}+b_0R_0+b_1R_1+\Delta,\]
where $a$, $b$,  and $b_i$ are non-negative rational numbers and $\Delta$ is an effective $\mathbb{Q}$-divisor whose support contains none of the curves $W$, $L_{xy}$, $R_0$, $R_1$. Lemma~\ref{lemma:estimate-ab} shows that $$ b<\frac{4}{5}; \ \  b_0, b_1 <\frac{3}{5}.$$

Suppose that the point $p$ lies on $L_{xy}$.
Since $\lambda b\leq 1$ and~$p\not\in W\cup R_0\cup R_1$, the log pair
\[\left(S_n, L_{xy}+\lambda\Delta\right)\]
is not log canonical at $p$. We then obtain a contradiction from  Lemma~\ref{lemma:inversion-of-adjunction} and the inequality
\[
\Delta\cdot L_{xy}=
\left( D-aW-bL_{xy}-b_0R_0-b_1R_1\right)\cdot L_{xy}
\leq\left( D-bL_{xy}\right)\cdot L_{xy} 
=\frac{3+2b(4n-1)}{4n(4n+1)}<\frac{1}{\lambda}.\]

We now suppose that the point $p$ lies on $R_i$.
Since $\lambda b_i\leq 1$ and~$p\not\in L_{xy}\cup W\cup R_j$, where $i\ne j$, the log pair
\[\left(S_n, R_i+\lambda\Delta\right)\]
is not log canonical at $p$. This also yields an absurd inequality 
\[\Delta\cdot R_i=
 \left( D-aW-bL_{xy}-b_0R_0-b_1R_1\right)\cdot R_i\leq (D\cdot R_i-R_i^2)=\frac{4n+5}{4(4n+1)}<\frac{1}{\lambda}.\]

Therefore, the point $p$ must be located outside the curves $C_x$.

Let $C$ be a curve in the pencil $|\mathcal{O}_{S_n}(2)|$ that passes through the point $p$. Since the curve $C$ is cut by $y=\gamma x^2$ for some constant $\gamma$,  it consists of two irreducible curves $L_{xy}$ and $R$. 
As in \eqref{equation:D-c}, we now may write
\[D=aW+bL_{xy}+cR+\Omega,\]
where $\Omega$ is an $\mathbb{Q}$-effective divisor whose support contains none of $W$, $L_{xy}$, $R$.
Lemmas~\ref{lemma:estimate-ab} implies that $$a<\frac{1}{8n}; \ \ b< \frac{4}{5}; \ \  c <\frac{3}{10}.$$

The log pair 
\[\left(S_n, \left(\frac{1}{2}+\lambda a\right)W+ \lambda bL_{xy}+\lambda cR+\lambda\Omega\right)\]
is not log canonical at $p$. 

Suppose that $p\not\in W$. Since   $\lambda c\leq 1$, the log pair
\[\left(S_n, R+\lambda\Omega\right)\]
is not log canonical at $p$ either. Lemma~\ref{lemma:inversion-of-adjunction} then implies an absurd inequality
\[\frac{1}{\lambda}< \Omega\cdot R=(D-aW-bL_{xy}-cR)\cdot R\leq D\cdot R=\frac{3}{4n}.\]
This means that the point $p$ must belong to $W$.
Then the log pair
\[\left(S_n, \left(\frac{1}{2}+\lambda a\right)W+R+\lambda\Omega\right)\]
is not log canonical at $p$

The curve $R$ meets $W$ at $p$ either transversally or tangentially. When they meet at $p$ tangentially, their local intersection number
at $p$ is  $2$.

We first consider the case when the curve $R$ meets $W$ at $p$ transversally. In this case, we can easily obtain a contradiction,
\[\begin{split}
1&<\left(\left(\left(\frac{1}{2}+\lambda a\right)W+\lambda\Omega\right)\cdot R\right)_p\\
&\leq \left(\frac{1}{2}+\lambda a\right)+\lambda\Omega\cdot R \\
&=\left(\frac{1}{2}+\lambda a\right)+\lambda (D-aW-bL_{xy}-cR)\cdot R\\
&\leq \frac{1}{2}+\lambda D\cdot R=\frac{1}{2}+\lambda\frac{3}{4n}<1
\end{split}\]
from Lemma~\ref{lemma:inversion-of-adjunction}.
Therefore, the curve $R$ meets $W$ at $p$ with local intersection number $2$. Note that
\[\mathrm{mult}_p(\Omega)\leq \Omega\cdot R\leq \frac{3}{4n}.\]

 Let $\phi:\widetilde{S}_n\to S_n$ be the blow up at $p$ and $E$ be its exceptional curve. Then
\[\phi^*\left(K_{S_n}+\left(\frac{1}{2}+\lambda a\right)W+ \lambda bL_{xy}+\lambda cR+\lambda\Omega\right)\]
\[=K_{\widetilde{S}_n}+\left(\frac{1}{2}+\lambda a\right)\widetilde{W}+ \lambda b\widetilde{L}_{xy}+\lambda c\widetilde{R}+\lambda\widetilde{\Omega}+dE,\]
where $\widetilde{W}$, $\widetilde{L}_{xy}$, $\widetilde{R}$, and $\widetilde{\Omega}$ are 
the proper transforms of $W$, $L_{xy}$, $R$, and $\Omega$ respectively. Here 
$d=\lambda(a+c)+\lambda \mathrm{mult}_p(\Omega)-\frac{1}{2}$.
Since $d\leq 1$ and $\lambda \mathrm{mult}_p(\Omega)\leq 1$, the log pair
\[\left(\widetilde{S}_n, \left(\frac{1}{2}+\lambda a\right)\widetilde{W}+ \lambda b\widetilde{L}_{xy}+\lambda c\widetilde{R}+\lambda\widetilde{\Omega}+dE\right)\]
is not log canonical at the point $q$ where $E$, $\widetilde{R}$, and $\widetilde{W}$ meet.
Let $\psi:\overline{S}_n\to\widetilde{S}_n$ be the blow up at the point $q$ and let $F$ be the exceptional curve of $\psi$.
Denote  the proper transforms of 
$\widetilde{W}$, $\widetilde{L}_{xy}$, $\widetilde{R}$, $\widetilde{\Omega}$, and $E$ by $\overline{W}$, $\overline{L}_{xy}$, $\overline{R}$,  $\overline{\Omega}$, and $\overline{E}$, respectively. 
Then
\[\psi^*\left(K_{\widetilde{S}_n}+\left(\frac{1}{2}+\lambda a\right)\widetilde{W}+ \lambda b\widetilde{L}_{xy}+\lambda c\widetilde{R}+\lambda\widetilde{\Omega}+dE\right)\]
\[=K_{\widetilde{S}_n}+\left(\frac{1}{2}+\lambda a\right)\overline{W}+ \lambda b\overline{L}_{xy}+\lambda c\overline{R}+\lambda\overline{\Omega}+d\overline{E}+eF,\]
where $e=\lambda (a+c)+d+\lambda \mathrm{mult}_q(\widetilde{\Omega})-\frac{1}{2}$.

Since $$e=2\lambda (a+c)+\lambda \left(\mathrm{mult}_p(\Omega)+\mathrm{mult}_q(\widetilde{\Omega})\right)-1\leq 
2\lambda \left(a+c+\mathrm{mult}_p(\Omega)\right)-1\leq 1,$$
the log pair 
\[ \left(\widetilde{S}_n, \left(\frac{1}{2}+\lambda a\right)\overline{W}+ \lambda b\overline{L}_{xy}+\lambda c\overline{R}+\lambda\overline{\Omega}+d\overline{E}+F\right)\]
is not log canonical at a point on $F$. Note that the curves  $\overline{W}$, $\overline{R}$, and $\overline{E}$ meet $F$ transversally at distinct points. However, the inequalities
\[\lambda \overline{\Omega}\cdot F=\lambda \mathrm{mult}_q(\widetilde{\Omega})\leq \lambda \mathrm{mult}_p(\Omega)\leq \frac{3\lambda}{4n}<1,\]
\[\left(\frac{1}{2}+\lambda a\right)+\lambda \overline{\Omega}\cdot F\leq \left(\frac{1}{2}+\lambda a\right)+\frac{3\lambda}{4n}<1,\]
\[\lambda c+\lambda \overline{\Omega}\cdot F\leq \lambda c+\frac{3\lambda}{4n}<1,\]
\[d+\lambda \overline{\Omega}\cdot F=\lambda(a+c)+2\lambda \mathrm{mult}_p(\Omega)-\frac{1}{2}<1\]
imply that 
the log pair above is log canonical  along the curve $F$ by from Lemma~\ref{lemma:inversion-of-adjunction}. This is a contradiction. Consequently, the log pair
$\left(S_n, \frac{1}{2}W+\lambda D\right)$ must be log canonical in the smooth locus of $S_n$.
\end{proof}

\begin{remark}
Another way to verify Theorem~\ref{theorem:smooth}  is  to apply a general version of Theorem~\ref{theorem:Hamid-Ziquan-Kento} as in the proof of Theorem~\ref{theorem:singular} below.  
We however use a direct method that is a bit simpler and demonstrate an instructive and basic approach to estimations of $\delta$-invariants.
\end{remark}

\begin{theorem}\label{theorem:singular}
For singular points $p=O_z$, $O_0$, and $O_1$, 
 $$\delta_p\left(S_n, \frac{1}{2}W\right)\geq \lambda.$$ \end{theorem}
\begin{proof}
For the singularity $O_z$, we obtain the log discrepancy 
\[A_{L_{xy}, \Delta}(O_z)=1-\ord_{O_z}(\Delta_L)=\frac{1}{4n}\]
from the adjunction formula
\[\left.\left(K_{S_n}+\frac{1}{2}W+L_{xy}\right)\right\vert_{L_{xy}}=K_{L_{xy}}+\Delta_L.\]
Then, \eqref{equation:S-L}, \eqref{equation:Kento-formula-S-L}, and Theorem~\ref{theorem:Hamid-Ziquan-Kento} imply that
\[\delta_{O_z}\left(S_n,\frac{1}{2}W\right)\geq\min\left\{\frac{2(2n+1)}{3n+1},\frac{(2n+1)(4n+1)}{4n^2+3n+1}\right\}>\lambda.\]

Similarly, for the singularity $O_i$, we obtain the log discrepancies
\[A_{R_i, \Delta}(O_i)=1-\ord_{O_i}\left(\Lambda_i\right)=\frac{1}{4}\]
from the adjunction formula
\[\left.\left(K_{S_n}+\frac{1}{2}W+R_i\right)\right\vert_{R_i}=K_{R_i}+\Lambda_i.\]
Theorem~\ref{theorem:Hamid-Ziquan-Kento} with \eqref{equation:S-Ri} and \eqref{equation:Kento-formula-S-Ri} then yields
\[\delta_{O_i}\left(S_n,\frac{1}{2}W\right)\geq\min\left\{\frac{4n(2n+1)}{4n^2+3n+1},\frac{2n(2n+1)(4n+1)}{8n^2+7n+1}\right\}>\lambda.\]
\end{proof}

\begin{theorem}\label{theorem:singular-w}
For the singular point $O_w$, 
 $$\delta_{O_w}\left(S_n, \frac{1}{2}W\right)\geq \lambda.$$ \end{theorem}
\begin{proof}
With a sufficiently large positive integer $m$, let $D$ be a $\mathbb{Q}$-divisor of $m$-basis type with respect to the log del Pezzo surface $(S_n, \frac{1}{2}W)$. It is enough to show that the log pair 
$$\left(S_n, \frac{1}{2}W+\lambda D\right)$$
is log canonical at $O_w$. Since the point $O_w$ is away from the curve $W$, we will prove that
$\left(S_n, \lambda D\right)$
is log canonical at $O_w$. 

Suppose that $\left(S_n, \lambda D\right)$
is not log canonical at $O_w$.  As \eqref{equation:D-ab}, we write
\[D=bL_{xy}+b_0R_0+b_1R_1+\Lambda,\]
where $b$, and $b_i$ are non-negative rational numbers and $\Lambda$ is an effective $\mathbb{Q}$-divisor whose support contains none of the curves $L_{xy}$, $R_0$, $R_1$.
By Lemma~\ref{lemma:estimate-ab},
$$b<\frac{3n+2}{2(2n+1)}; \ \  b_0, b_1 <\frac{8n^2+6n+3}{8n(2n+1)}.$$

Let $\phi:\hat{S}_n\to S_n$ be the weighted blow up at $O_w$ with weights $(1,n )$ and $F$ be its exceptional curve.
Then
\[K_{\hat{S}_n}=\phi^*(K_{S_n})-\frac{3n}{4n+1}F.\]
Denote the proper transforms of $L_{xy}$, $R_0$, $R_1$, and $\Lambda$ by $\hat{L}_{xy}$, $\hat{R}_0$, $\hat{R}_1$, and $\hat{\Lambda}$, respectively. 

The exceptional curve $F$ contains one singular point of $\hat{S}_n$, where $F$ and $\hat{L}_{xy}$ intersect. It is a cyclic quotient singularity of type $\frac{1}{n}(-1, 1)$.

We have
\[ \hat{L}_{xy}=\phi^*(L_{xy})-\frac{1}{4n+1}F, \  \  \  \hat{R}_{i}=\phi^*(R_i)-\frac{n}{4n+1}F,   \  \  \   \hat{\Lambda}=\phi^*(\Lambda)-\frac{\mu}{4n+1}F,\]
where $\mu$ is a non-negative rational number, and hence
\[K_{\hat{S}_n}+ \lambda \left(b\hat{L}_{xy}+ b_0\hat{R}_{0}+ b_1\hat{R}_{1}+\hat{\Lambda}\right)+\left(\frac{3n}{4n+1}+\lambda\theta\right)F=\phi^*(K_{S_n}+\lambda D),\]
where
\[\theta=\frac{b+n(b_0+b_1)+\mu}{4n+1}.\]

Since $F^2=-\frac{4n+1}{n}$,  we obtain
\[\hat{L}_{xy}^2=-\frac{1}{2n},\ \ \ \hat{R}_0^2=\hat{R}_1^2=-\frac{1}{2}, \ \ \ \hat{L}_{xy}\cdot\hat{R}_0=\hat{L}_{xy}\cdot\hat{R}_1=\hat{R}_0\cdot\hat{R}_1=0,\]
    \[\hat{L}_{xy}\cdot F=\frac{1}{n},\ \ \ \hat{R}_0\cdot F=\hat{R}_1\cdot F=1.\]

For the estimation of $\theta$, we first compute the volume of $\phi^*(D) - tF$.
Since  $\hat{L}_{xy}$, $\hat{R}_0$, $\hat{R}_1$ are of negative self-intersection, and 
\[
    \phi^*(D) - tF \equiv \phi^\ast\left(\frac{3}{2}C_x\right) - tF= \frac{3}{2}\left(\hat{L}_{xy}+\hat{R}_0+\hat{R}_1\right)+ \left(\frac{3(2n + 1)}{2(4n + 1)} - t\right)F,
\]
for $t>  \frac{6n+3}{8n+2}$, the divisor  $\phi^*(D) - tF$ is not pseudoeffective. Put

\[\begin{split}
    P_F(t)&=\left\{\aligned
    &\frac{3}{2}\left(\hat{L}+\hat{R}_0+\hat{R}_1\right)+\left(\frac{3(2n+1)}{2(4n+1)}-t\right)F\ \ \ \text{ for $0\leq t\leq\frac{3}{4(4n+1)}$},\\
    &\\
    &\left(\frac{3(2n+1)}{2(4n+1)}-t\right)(2\hat{L}+2\hat{R}_0+2\hat{R}_1+F)\ \ \ \text{ for $\frac{3}{4(4n+1)}\leq t\leq\frac{3(2n+1)}{2(4n+1)}$}.
    \endaligned\right.\\
    &\\
    N_F(t)&=\left\{\aligned
    &0\ \ \ \text{ for $0\leq t\leq\frac{3}{4(4n+1)}$,}\\
    &\\
    &\left(2t-\frac{3}{2(4n+1)}\right)(\hat{L}+\hat{R}_0+\hat{R}_1)\ \ \ \text{ for $\frac{3}{4(4n+1)}\leq t\leq\frac{3(2n+1)}{2(4n+1)}$}.
    \endaligned\right.
\end{split}\]
For $0\leq t\leq\frac{3}{4(4n+1)}$, 
\[
    P_F(t)\cdot\hat{L}_{xy}=\frac{1}{n}\left(\frac{3}{4(4n+1)}-t\right),\ \ P_F(t)\cdot\hat{R}_0=P_F(t)\cdot\hat{R}_1=\frac{3}{4(4n+1)}-t.
\]
For $\frac{3}{4(4n+1)}\leq t\leq\frac{3(2n+1)}{2(4n+1)}$, 
$$P_F(t)\cdot\hat{L}_{xy}=P_F(t)\cdot\hat{R}_0=P_F(t)\cdot\hat{R}_1=0.$$
Therefore, the divisor $P_F(t)$ is nef. The Zariski decomposition of $\phi^\ast\left(D\right)-tF$ is given by
\[P_F(t)+N_F(t).\]
Thus, the volume is given by
\[
    \vol(\phi^\ast(D)-tF)=\left\{\aligned
    &\frac{9(2n+1)}{8n(4n+1)}-\frac{4n+1}{n}t^2\ \ \ \text{ for $0\leq t\leq\frac{3}{4(4n+1)}$,}\\
    &\\
    &\frac{1}{n}\left(\frac{3(2n+1)}{2(4n+1)}-t\right)^2\ \ \ \text{ for $\frac{3}{4(4+1)}\leq t\leq\frac{3(2n+1)}{2(4n+1)}$},
    \endaligned\right.
\]

so the value in \eqref{equation:S} is given by
\[S_{S, \frac{1}{2}W}(F)=   \frac{1}{D^2}\int_0^{\infty} \vol(\phi^*(D) - tF) dt=\frac{4n+3}{4(4n+1)}.\]
Thus, it follows from Lemma~\ref{lemma:Fujita} that for a sufficiently small positive real number $\epsilon$
\begin{equation}\label{bound of theta}
 \frac{b+n(b_0+b_1)+\mu}{4n+1}= \theta < \frac{4n+3}{4(4n+1)} + \epsilon.
\end{equation}
It implies that
\[
    \frac{3n}{4n+1} + \lambda\theta < 1.
\]
Therefore, the log pair
\[\left( \hat{S}_n, \lambda b\hat{L}_{xy}+\lambda b_0\hat{R}_{0}+\lambda b_1\hat{R}_{1}+\lambda\hat{\Lambda}+F\right)\]
is not log canonical at some point $q$ on $F$.

We first suppose  that $q\in F\setminus \hat{L}_{xy}\cup \hat{R}_0\cup \hat{R}_1$. Then the log pair $(\hat{S}, \lambda \hat{\Lambda} + F)$ is not log canonical at the point $q$. Lemma~\ref{lemma:inversion-of-adjunction} then implies
\[
    \frac{1}{\lambda} < \hat{\Lambda}\cdot F = \frac{\mu}{n}.
\]
However, the inequality 
\[0\leq \hat{\Lambda}\cdot \hat{R}_i = \frac{3}{4(4n+1)} - \frac{b}{4n+1} + \frac{b_i(2n+1)}{2(4n+1)} - \frac{b_j n}{4n + 1} - \frac{\mu}{4n + 1}, \]
where the indices $i, j$ are determined by $b_i \leq b_j$ with $\{i, j\}=\{0,1\}$, 
yields the opposite inequality
\[
    \mu \leq \frac{3}{4} - b -b_j n + \frac{b_i(2n+1)}{2} \leq \frac{3}{4} + \frac{b_i}{2}\leq \frac{3}{4} + \frac{8n^2 + 6n + 3}{16n(2n+1)} < \frac{n}{\lambda}.
\]
Therefore, the point $q$ must be one of the intersection points  $F\cap\hat{R}_0$, $F\cap\hat{R}_1$, $F\cap\hat{L}_{xy}$.

We first consider the case when $q$ is the intersection point of $F$ and $\hat{R}_i$. Then 
the log pair $(\hat{S}, \lambda (\hat{\Lambda} + b_i\hat{R}_i) + F)$ is not log canonical at $q$. We then have
\[\frac{1}{\lambda} < (\hat{\Lambda} + b_i\hat{R}_i)\cdot F = \frac{\mu}{n} + b_i.\]

From \eqref{bound of theta}
we obtain 
\[\frac{1}{\lambda}+b_j<\frac{b}{n}+(b_0+b_1)+\frac{\mu}{n}<\frac{4n+3}{4n} + \epsilon.\]

On the other hand, from 
\begin{equation}\label{eq:R-j}
0\leq \hat{\Lambda}\cdot \hat{R}_j = \frac{3}{4(4n+1)} - \frac{b}{4n+1} + \frac{b_j(2n+1)}{2(4n+1)} - \frac{b_i n}{4n + 1} - \frac{\mu}{4n + 1}
\end{equation}
we obtain
\[
    \frac{\mu}{n} + b_i \leq \frac{3}{4n} - \frac{b}{n} + \frac{b_j(2n+1)}{2n}.
\]
Then
\[\frac{1}{\lambda}=\frac{20n+4}{20n+5}<\frac{\mu}{n} + b_i \leq \frac{3}{4n} + \frac{b_j(2n+1)}{2n},\]
and hence
\[\frac{80n^2-44n-15}{10(2n+1)(4n+1)}<b_j.\]
This yields a contradictory inequality
\[
\frac{1}{\lambda}+b_j>\frac{20n+4}{20n+5}+\frac{80n^2-44n-15}{10(2n+1)(4n+1)}=\frac{40n-7}{20n+10}.
\]
Consequently, $q$ must be the intersection point of $F$ and $\hat{L}_{xy}$, which is a singular point of type~$\frac{1}{n}(-1, 1)$.

Then the log pair $(\hat{S}, \lambda (\hat{\Lambda} + b\hat{L}_{xy}) + F)$ is not log canonical at $q$. We then obtain
\[\frac{1}{n\lambda} < (\hat{\Lambda} + b\hat{L}_{xy})\cdot F = \frac{\mu}{n} + \frac{b}{n}\]
from Lemma~\ref{lemma:inversion-of-adjunction}.
Meanwhile, if $b_j\leq b_i$,  we use  \eqref{eq:R-j} to obtain 
\[
    4(b + \mu)-3 \leq b_0 + b_1.
\]
Together with  \eqref{bound of theta}
this implies that
\[
    \frac{b}{4n + 1} + \frac{n}{4n + 1}(4(b + \mu) - 3) + \frac{\mu}{4n + 1} = (b + \mu) - \frac{3n}{4n + 1}< \frac{4n + 3}{4(4n + 1)}+ \epsilon.
\]
Thus
\[
    \frac{20n+4}{20n+5}=\frac{1}{\lambda} < b + \mu \leq \frac{16n + 3}{16n + 4}+ \epsilon.
\]
This is absurd. 

Therefore, we may conclude that
the log pair $\left(S_n, \frac{1}{2}W+\lambda D\right)$ is log canonical at $O_w$.
\end{proof}

\begin{proof}[Proof of Theorem~\ref{theorem:KE-del}]
Theorems~\ref{theorem:smooth}, \ref{theorem:singular}, and~\ref{theorem:singular-w}
immediately imply 
\[\delta\left(S_n,\frac{1}{2}W\right)>1.\]
Then Theorems~\ref{theorem:delta} and~\ref{theorem:delta-KE} complete the proof.
\end{proof}

\medskip

\textbf{Acknowledgements.}
D.~Jeong and J.~Park have been supported by IBS-R003-D1, Institute for Basic Science in Korea. 
I.~Kim and J.~Won were supported by NRF grant funded by the Korea government(MSIT) (I.~Kim: NRF-2020R1A2C4002510,
J.~Won: NRF-2020R1A2C1A01008018).

\end{document}